\documentclass[leqno,a4paper]{amsart}
\usepackage[english]{babel} 
\usepackage{hyperref} 
\newtheorem{definition}{Definition}[section]
\newtheorem{theorem}{Theorem}[section]
\newtheorem{proposition}{Proposition}[section]
\newtheorem{lemma}{Lemma}[section]
\newtheorem{remark}{Remark}[section]
\newcommand{\nada}[1]{}
\newcommand{\R}{\mathbb{R}} 
\newcommand{\T}{\mathbb{T}}
\newcommand{\N}{\mathbb{N}} 
\newcommand{\Z}{\mathbb{Z}}
\newcommand{\dom}{\mathcal{D}} 
\newcommand\D{\partial}
\newcommand\X{\times}
\renewcommand{\div}{\mathop\mathrm{div}\nolimits}

\newcommand{\bbun}{{u}^m} 
\newcommand{\ut}{u(t)}
\newcommand{\us}{u(s)}

 \numberwithin{equation}{section}
\begin{document}

\date{\today}

\title[Uniqueness of particle trajectories for  non-Newtonian
2D fluids] {An elementary proof of uniqueness of the particle
  trajectories for solutions of a class of  shear-thinning
  non-Newtonian 2D fluids}

\author[L. C. Berselli]{Luigi C. Berselli}

\address{(Luigi C. Berselli) Dipartimento di Matematica Applicata
  ``U.~Dini,'' Universit\`a di Pisa, Via F.~Buonarroti 1/c, I-56127,
  Pisa, Italia}

\email{\href{mailto:berselli@dma.unipi.it}{berselli@dma.unipi.it}}

\urladdr{\url{http://users.dma.unipi.it/berselli}}

\author[L. Bisconti]{Luca Bisconti}

\address{(Luca Bisconti) 
Dipartimento di Sistemi e Informatica, Universit\`a degli Studi di Firenze,
Via S.~Marta~3, I-50139 Firenze, Italia}

\email{\href{mailto:luca.bisconti@unifi.it}{luca.bisconti@unifi.it}}

\date{\today}

\maketitle

\begin{abstract}
  We prove some regularity results for a class of two dimensional
  non-Newtonian fluids. By applying results from~[Dashti and Robinson,
  \textit{Nonlinearity}, 22 (2009), 735-746] we can then show
  uniqueness of particle trajectories.
\end{abstract}
\section{Introduction}
In this paper we consider the following system of partial differential
equations
\begin{subequations} 
  \label{eq:genNewModel-m}
  \begin{align}
    u_t -\nu_0\Delta u -\nu_1 \div S (\dom u)+ ( u \cdot \nabla)\, u
    + \nabla \pi = f\quad \, \textrm{in}\quad \,& [0,T] \X \Omega,
    \label{eq:genNewMod1-m}
    \\
    \div \, u = 0 \quad \,\textrm{in}\quad \, &[0,T]\X \Omega,
    \label{eq:genNewMod2-m} 
    \\
    u (0) = u_0 \quad \, \textrm{in}\quad \, &\Omega,
    \label{eq:genNewMode3-m}
  \end{align}
\end{subequations}
where $\Omega$ denotes either a two-dimensional bounded domain or the
two dimensional flat torus, the vector field $ u=(u_1, u_2)$ is the
velocity, the scalar $\pi$ is the kinematic pressure, the vector
$f=(f_1, f_2)$ is the external body force, $ u_0$ is the initial
velocity, and $\nu_0,\,\nu_1$ are positive constants. We denote by
\begin{equation*}
  \dom u:=\frac{1}{2}(\nabla  u + \nabla  u^T)=
  \frac{1}{2}(\D_j u_i + \D_i u_j)\quad \text{for } i,j=1,2,
\end{equation*}
the symmetric part of $\nabla u$, the convective term is
$(u\cdot\nabla)\, u:=\sum_{k=1}^2u_k\D_k u$, and $S$
% \colon\R^{2\X 2}\to\R^{2\X 2}_{sym}$ 
denotes the extra stress tensor, defined by
\begin{equation}
  \label{eq:extra-stress-tensor}
  S (\dom u):=(\delta +
  |\dom u| )^{p-2}\dom u, \qquad p\in [1,2),
\end{equation}
where $\delta$ is a non-negative
constant. System~\eqref{eq:genNewModel-m} describes a shear-thinning
homogeneous fluid and for an introduction to the mathematical theory
see M{\'a}lek, Rajagopal, and R{\r u}{\v{z}}i{\v{c}}ka~\cite{MRR1995}.
We mainly study the problem, endowed with homogeneous Dirichlet
boundary conditions
\begin{equation} 
  \label{eq:bound-cond-fake} 
  u_{\vert_{\Gamma}}=0\quad  \textrm{where}\quad \Gamma=\D\Omega,
\end{equation}
but we give some remarks also on the periodic case.

The main goal of this paper is to study the problem of uniqueness for
the \textit{particle trajectories} (or characteristics), which are
solutions of the following Cauchy problem
\begin{equation} 
  \label{eq:traj}
  \left\{
    \begin{aligned}
      &\frac{d {X}(t)}{dt} = u({X}(t), t)\qquad t\in[0,T],
      \\
      &{X}(0) = x\in \Omega,
    \end{aligned}
  \right. 
\end{equation}
where $u$ is the fluid velocity in~\eqref{eq:genNewModel-m}.  For the
3D~Navier-Stokes equations the problem of existence of particle
trajectories and Lagrangian representation of the flow started with
the work of Foias, Guillop{\'e}, and Temam~\cite{FGT1985}, and related
results of regularity in $\R^n$ are proved in Chemin and
Lerner~\cite{CL1995} by means of Littlewod-Paley decomposition. The
question of uniqueness has been recently addressed by elementary tools
and in a more general context in Robinson \textit{et
  al.}~\cite{DA-RO:2009,RS2009,RS2009b} and it is strictly related
with uniqueness for linear transport equations.  We consider here the
same problem, in the case of shear-thinning fluids, described
by~\eqref{eq:genNewModel-m}. To this end, we will study certain
regularity properties of the solutions of~\eqref{eq:genNewModel-m},
investigating when the velocity will verify the appropriate hypotheses
for uniqueness results.

In particular, classical results concerning Lipschitz continuous
fields $u$ (which generally can be verified checking that $\nabla u$
is bounded in the space variables) are not easily applicable here,
since such a regularity is very difficult to be proved, even in the
two dimensional case, for~\eqref{eq:genNewModel-m}. We recall that,
\textit{restricting to the two dimensional case}, some
$C^{1,\gamma}$-results are obtained in Kaplick{\'y}, M{\'a}lek, and
Star{\'a}~\cite{KMS1997,KMS1999} in the stationary case. Early results
in the time dependent case (but not up-to-the-boundary) are those by
Seregin~\cite{Ser1997}, while results in the space-periodic
time-dependent case have been obtained in~\cite{KMS2002}. We observe
that essentially all the above results require the extra-stress tensor
$S$ to be slightly smoother than that
in~\eqref{eq:extra-stress-tensor}.  In particular, it is requested
that the stress-tensor is replaced, for instance, by $S(\dom
u)=(\delta+|\dom u|^2)^{\frac{p-2}{2}}\dom u$. In any case we study
the regularity up-to-the-boundary with non-smooth initial data and our
results, proved in an elementary way, are original. Moreover, the
difficulties appearing in the 3D~case seem completely out of the
current mathematical knowledge for such fluids, and this explains why
we restrict to the two dimensional case.

Since we want to have \textit{elementary} proofs (in order to possibly
extend the results to the widest possible class of solutions and
stress-tensors) we will work with the classical energy-type
methods. Concerning uniqueness of particle trajectories, there have
been some recent improvements, strictly related with the Osgood
criterion and with Log-Lipschitz properties of Sobolev functions
$W^{s+1,q}(\R^d)$ in the case of limiting Sobolev exponents such that
$q=\frac{d}{s}$. In particular we will use the result below,
{proved in~\cite[Theorem 2.1]{DA-RO:2009}}.
\begin{theorem} 
  \label{thm:Dashti-Robinson} 
  Let $\Omega$ be either the whole space $\R^d$, $d\geq 2$, a
  periodic $d$-dimensional domain, or an open bounded subset of $\R^d$
  with a sufficiently smooth boundary. Let assume that for some $p>1$
  \begin{equation*}
    u\in L^p(0,T; W^{\frac{d-2}{2}, 2}(\Omega))\quad 
    \textrm{and}\quad 
    \sqrt{t}\,u\in L^2(0,T;  W^{\frac{d+2}{2}, 2}(\Omega)),
  \end{equation*}
  with $u_{|\Gamma} = 0$, when $\Omega$ is a domain with boundary. Then,
  the Cauchy problem~\eqref{eq:traj} has a unique solution in $[0,T]$.
\end{theorem}
The latter result shows that certain (slightly weaker than
$C^{1,\gamma}$) results of Sobolev space-regularity can be used to
obtain uniqueness for~\eqref{eq:traj}. On the other hand, the
$W^{2,2}(\R^2)$ regularity for fluid with shear-dependent viscosity is
another non-trivial task (while in 3D proving $u\in W^{5/2,2}(\R^3)$,
seems at the moment out of sight).  Some recent results (in the
stationary case) for second-order space-derivatives appeared
in~\cite{BEI:2009,Ber2009a,CG2008} even if the square integrability of
second order derivatives is not reached in general domains, or if
certain limitations on the smallness of the force are not
satisfied. For the non-stationary case, we recall the result in the
space periodic setting (obtained uniformly in $\delta\geq0$)
from~\cite{BE-DI-RU:2010,DI-RU:2005}.

We also point out that one of the main technical obstructions is
represented by the pressure and the associated divergence-free
constraint. In the case of the p-Laplacian systems, in fact, the recent
results in Beir\~ao da Veiga and Crispo~\cite{BeiC2012} show that
$u\in W^{2,q}(\Omega)$, for arbitrary $q$, if $f$ is smooth, and under
certain restrictions on the range of $p\in(1,2)$. These latter results
are proved in the stationary case, they have no counterpart for the
$p$-Stokes system, and most likely they can be adapted to the
time-dependent case.%, see~\cite{BeiC2012b}

We point out that in the case of non-Newtonian fluids many features of
the problem are critical: The type of boundary conditions, the range
of $p$, and if the parameter $\delta$ is strictly larger than zero.
We will discuss later on some of the technical issues of the problem
and we will explain why we have to reduce to the 2D case with
$\nu_0,\delta>0$.  We start by considering the easier case of the
periodic setting where $\Omega$ is the flat 2D~torus
$\T^2:=\R^2/2\pi\Z$ and we will prove the following result.
\begin{proposition} 
  \label{prop:unicita-caso-periodico} 
  Let $\nu_0>0$, $\delta \geq 0$, and $p\in(1,2]$.  Let be given
  $u_0\in L^2(\T^2)$ such that $\div u_0=0$ and $f\in
  L^2(0,T;L^2(\T^2))$. Then, weak solutions
  to~\eqref{eq:genNewModel-m} satisfy $\sqrt{t}\,u\in
  L^2(0,T;W^{2,2}(\T^2))$ and hence problem~\eqref{eq:traj} admits a
  unique solution.
\end{proposition}
We emphasize that the assumption $\nu_0 > 0$ is crucial in our
method. When $\nu_0=0$ it is possible to prove a regularity result
that, although it is not useful to get an application of
Theorem~\ref{thm:Dashti-Robinson}, seems interesting by itself.  See
Prop.~\ref{eq:prop-nuzero}, cf.~Kost~\cite{Kos2010}.

In the Dirichlet case the problem of regularity is more delicate. We
will consider problem~\eqref{eq:genNewModel-m} in a domain with flat
boundary. We first prove a regularity result, by using techniques
similar to those used in~\cite{CR:2009} and formerly introduced, for
the case $p>2$, in~\cite{Bei2005}. With smooth data, we have the
following result.
\begin{proposition} 
  \label{lem:main-theo-with-delta-with-nonlin}
  Let $\delta>0$, $\nu_0>0$, $p\in \big[\frac{3}{2}, 2\big]$,
  $u_0\in W^{2,2}(\Omega)\cap W^{1,2}_0(\Omega)$ with $\div u_0=0$, and
  $f\in W^{1,2}(0,T;L^2(\Omega))$. Then,  weak solutions to
  Problem~\eqref{eq:genNewModel-m}-\eqref{eq:bound-cond-fake} satisfy
 \begin{equation}
    \label{eq:regularity-maximal}
    \begin{aligned} \|u_t\|_{L^{\infty}(0, T;L^2)} &+ \|\nabla
      u \|_{L^{\infty}(0, T; L^2(\Omega))}+\|\nabla \pi \|_{L^2(0,T;L^2)}   
      \\
      & + \|\nabla u_t\|_{L^2(0,T;L^2)} + \|D^2
      u\|_{L^2(0,T;L^2)} \leq C,
    \end{aligned}
  \end{equation}
  where $C$ depends on $p$, $\delta$, $\nu_0$, $\nu_1$,
  $\|f\|_{W^{1,2}(0, T;L^2)}$, $\|u_0\|_{2,2}$, $T$, and
  $\Omega$.  
\end{proposition}
Some hypotheses can be relaxed, since the time regularity is
unnecessary for the proof of uniqueness of particle trajectories, but
the arguments used to prove
Proposition~\ref{lem:main-theo-with-delta-with-nonlin} will play a
fundamental role to demonstrate our main uniqueness criterion for the
problem~\eqref{eq:traj}.  The main result of this paper reads as
follows:
\begin{theorem} 
  \label{thm:teoUniq} 
  Let $\delta>0$, $\nu_0>0$, $p\in\big[\frac{3}{2},2]$, $u_0\in
  L^2(\Omega)$ with $\div u_0=0$, such that $(u_0\cdot
  n)_{|\Gamma}=0$, and $f\in L^2(0,T;L^2(\Omega))$. Then, weak
  solutions to~\eqref{eq:genNewModel-m}-\eqref{eq:bound-cond-fake}
  satisfy $\sqrt{t}\,u\in L^2(0,T;W^{2,2}(\Omega))$, and
  consequently~\eqref{eq:traj} admits a unique solution.
\end{theorem}

\noindent\textbf{Plan of the paper.}  In Section~\ref{sec:preliminaries}
we introduce the notation and we give some preliminary results. In
Section~\ref{periodic-case-section}, we consider the space-periodic
setting and we prove Proposition~\ref{prop:unicita-caso-periodico}.
Thereafter, in Section~\ref{sec:time-regularity}, we prove a
preliminary space-time regularity result for the solutions
of~\eqref{eq:genNewModel-m}-\eqref{eq:bound-cond-fake} and then we
demonstrate Proposition~\ref{lem:main-theo-with-delta-with-nonlin}.
Finally, in Section~\ref{sec:uniqueness}, we give the proof of
Theorem~\ref{thm:teoUniq}.
\section{Preliminaries and basic results}
\label{sec:preliminaries} 
Let us introduce the notation related especially to the
problem~\eqref{eq:genNewModel-m} with Dirichlet boundary conditions.
The needed assumptions or changes for the space periodic case are
specified in Section~\ref{periodic-case-section}.

Throughout the article, when $\Omega$ is a bounded domain with
boundary, it will be a two dimensional cube $\Omega=]-1,1[^2$ and we
denote by $\Gamma$ the two opposite sides in the $x_2$ direction
\begin{equation*}
  \Gamma:=\{x=(x_1, x_2) \colon |x_1|<1,\, x_2=-1 \}\cup
  \{x=(x_1, x_2) \colon |x_1|<1,\, x_2=1 \},
\end{equation*}
We use the following boundary conditions
\begin{equation}
  \label{eq:boundary-cond}
  \left\{\begin{array}{l}
      u_{\vert_\Gamma}=0,
      \\
      u\quad  \textrm{is}\quad  
      2\textrm{-periodic w.r.t } x_1.
    \end{array}\right.
\end{equation}
Here, $x_1$ represents the tangential direction to $\Gamma$ and this
idealized setting of a ``periodic strip'' corresponds to the
half-space, but without complications at infinity. 

Given $q\geq 1$, by $L^q(\Omega)$, we indicate the usual Lebesgue
space with norm $\|\cdot\|_q$.  Moreover, by $W^{k,q}(\Omega)$, $k$ a
non-negative integer and $q$ as before, we denote the usual Sobolev
space with norm $\|\cdot\|_{k,q}$. We also denote by
$W^{1,q}_0(\Omega)$ the closure of $C_0^{\infty}(\Omega)$ in
$W^{1,q}(\Omega)$ and by $W^{-1, q'}(\Omega)$, $q' = q/(q -1)$, the
dual of $W^{1, q}_0(\Omega)$ with norm $\|\cdot\|_{-1, q'}$.  Let $X$ be
a real Banach space with norm $\|\cdot\|_X$. We will use the customary
spaces $W^{k,q}(0, T;X)$, with norm denoted by
$\|\cdot\|_{W^{k,q}(0,T;X)}$, recalling that $W^{0,q}(0, T;X)=L^{q}(0,
T;X)$. We will also use the notation $\Omega_T:=\Omega\times(0,T)$ and
we will not distinguish between scalar and vector fields and the symbol
$\langle\, \cdot\,,\, \cdot\,\rangle$ will indicate a duality pairing.
Here and in the sequel, we denote by $C$ positive constants that may
assume different values, even in the same equation. We also define
\begin{equation*}
  V_q :=\big\{  v\in W^{1,q}(\Omega):\  \nabla \cdot v = 0,\,
  v_{\vert_\Gamma}= 0,\, v\ \textrm{is}\ 2\textrm{-periodic w.r.t. }x_1\big\}, 
\end{equation*}
with dual space $V'_q$.  Since the extra-stress tensor $S$ is a
function not of the gradient, but of the deformation tensor (in order
to have frame invariant equations) we recall a Korn-type inequality,
see~\cite{CR:2009}
\begin{lemma} 
  \label{lem:lemma-2.2-crispo} There exists a positive constant
  $C=C(\Omega)$ such that
  \begin{equation*}
    \|v\|_q + \|\nabla v \|_q \leq C \|\dom v\|_q, \quad  
    \textrm{for each}\quad  v\in V_q.
  \end{equation*}
%  Moreover, applying the above inequality to the vector field $u=\D_1
%  v$, we infer that
%  \begin{equation}
%    \label{eq:utility-utility} 
%    \|u\|_q+    \|\D_1 \nabla u\|_q      \leq C \|\D_1 \dom u\|_q.
%  \end{equation}
\end{lemma}
We write $f \simeq g$, if there exist $c_0,\,c_1>0$ such that $c_0 f
\leq g \leq c_1 f$. When considering the tensor
$S(D)=(\delta+|D^{sym}|)^{p-2}D^{sym}$, introduced
in~\eqref{eq:extra-stress-tensor} (where $D$ is a second order tensor
and $D^{sym}$ its symmetric part) it can be easily checked that for
any second order tensor $C$, the following relations are verified
\begin{subequations}
  \begin{gather}
    \sum_{i,j,k,l=1}^2\D_{kl} {S}_{ij}( D) C_{ij} C_{kl}
    \geq (p-1)(\delta + |D^{sym}|)^{p-2}|{C}|^2, \label{eq:ass-1a}
    \\
    \left| \D_{kl} {S}_{ij}( D)\right| \leq (3-p)(\delta +
    |D^{sym}|)^{p-2}.\label{eq:ass-1b}
  \end{gather}
\end{subequations}
The symbol $\D_{kl} {S}_{ij}$ represents the partial derivative $\D
S_{ij}/\D D_{kl}$ of the $(i,j)$-component of $S$ with respect to the
$(k, l)$-component of the underlying space of $2\X 2$
matrices. Monotonicity and growth properties of $S$ are characterized
in the following standard lemma.
\begin{lemma} 
  \label{eq:lemma-S} Assume that $p\in (1,\infty)$ and
  $\delta\in [0,\infty)$.  Then, for all $A, B \in \R^{2\X2}$ there
  holds
  \begin{subequations}
    \begin{gather*}
      \big( S(A) - S(B) \big) \cdot (A^{sym}- B^{sym}) \simeq (\delta +
      |B^{sym}| + |A^{sym}|)^{p-2} |A^{sym} - B^{sym}|^2,
      \\
      |S(A) - S( B)| \simeq (\delta + |B^{sym}| + |A^{sym}|)^{p-2}
      |A^{sym} - B^{sym}|,
    \end{gather*}
  \end{subequations}
  where the constants $c_0,\,c_1>0$ depend only on $p$, and are
  independent of $\delta \geq 0$.
\end{lemma}
From the elementary inequality $ a^p \leq a^2b^{p-2} + b^p$, valid for
all $0 \leq a$, $0 < b$, and $p \in [1, 2]$, we get the relation
\begin{equation} 
  \label{eq:2.20}
  \delta^{\frac{p}{2}} + t^{\frac{p}{2}} \simeq (\delta + t)^{\frac{p-2}{2}} t 
  + \delta^{\frac{p}{2}}, \qquad \delta,t\, \geq 0
\end{equation}
with constants depending only on $p$
(see~\cite[Corollary~2.19]{BE-DI-RU:2010}).

Since in the Dirichlet case we need to handle in a different way
tangential and normal derivatives, we denote by $D^2 u$ the set of all
the second-order partial derivatives of $u$. In addition, the symbol
$D^2_\ast u$ denotes all partial derivatives $\D_{ik}^2u_j$, except
for the derivative $\D_{22}^2u_1$, namely
\begin{equation*}
  |D^2_\ast u|^2 := |\D_{22}u_2|^2 + 
  \sum^2_{\fontsize{6}{4}\selectfont 
\left. 
      \begin{array}{c} i,j,k =1
        \\
        (i,k)\ne (2,2)\end{array}\right.}
  |\D_{ik}^2u_j|^2.
\end{equation*}

We introduce the following quantities strictly related to the stress
tensor $S$ and coming naturally in the problem, when using the
techniques introduced in~\cite{BEI:2009,DI-RU:2005,MRR1995}:
\begin{subequations}
  \begin{align}
    \mathcal{I}_1(u)&:= \int_{\Omega} (\delta + |\dom u|)^{p-2}|
    \D_{1}\dom u|^2 dx, \label{eq:I1}
    \\
    \mathcal{I}(u)&:= \int_{\Omega} (\delta + |\dom u|)^{p-2}| \nabla
    \dom u|^2 dx, \label{eq:I}
    \\
    \mathcal{J}(u)&:= \int_{\Omega} (\delta + |\dom u|)^{p-2}|\dom
    u_t|^2dx, \label{eq:J}
  \end{align}
\end{subequations}
where $\mathcal{I}$ is obtained by integration by parts when testing
the extra stress-tensor $S$ with $-\Delta u$ (and this is possible in
the periodic-case); a multiple of $\mathcal{I}_1$ is obtained testing
with $-\partial_{11} u$ and the calculations are possible in the flat
domain; Finally a multiple of $\mathcal{J}$ is obtained testing with
$u_{tt}$ and calculations are valid also in the Dirichlet case, for a
generic domain.

%The following result, which is a
%consequence of~\eqref{eq:utility-utility}, is taken
%from~\cite[Lemma~3.2]{BEI:2009}.
%%
%\begin{lemma} 
%  \label{lem:estimate-for-the-tangential-derivatives} For all
%  sufficiently smooth vector fields there is a positive constant
%  $C=C(\Omega)$ such that
%  \begin{equation*}
%    \|D^2_\ast u\|_q \leq c\big(\|\D_1 \nabla u\|_q + 
%    \|\D_1 u\|_q\big) \leq C \|\D_1 \dom u\|_q.
%  \end{equation*}
%\end{lemma}
We will also use this classical result, see Ne\v{c}as~\cite{Nec1966}.
\begin{lemma}
  \label{lem:lemma-2.6-crispo} If it holds
  $\nabla g =\div G$, for some $G \in (L^q(\Omega))^{2\times 2}$, for  $1<q<+\infty$ then
 \begin{equation*}
    \Big\|g-\int_\Omega{g}(x)\,dx\Big\|_q\leq c \|G\|_q.
  \end{equation*}
\end{lemma}

\medskip

Let us recall the definition of \emph{weak solution} to the
Problem~\eqref{eq:genNewModel-m}-\eqref{eq:boundary-cond}.

\begin{definition} 
  \label{def:defin-weak-sol} Let $T>0$ and assume that $f\in
  L^2(0,T;V'_2)$.  We say that $u$ is a weak solution of
  problem~\eqref{eq:genNewModel-m} if:
  \begin{subequations} \label{eq:weak-solution}
    \begin{gather}
      u \in L^2(0, T; V_2)\cap L^{\infty}(0, T; L^2(\Omega)),
      \label{eq:weak-solution-1}
      \\
      u_t\in L^2(0, T; V'_2), 
      \label{eq:weak-solution-2}
      \\
      \begin{aligned}
        &\int_\Omega u(t)\,\varphi\,dx + \nu_0
        \int^t_{0}\int_\Omega\nabla u(s)\, \nabla \varphi\,dx ds +\nu_1
        \int^t_{0}\big\langle S(\dom u(s)), \dom \varphi \big\rangle\,
        ds
        \\
        &\quad - \int^t_{0} \int_\Omega (u(s)\cdot \nabla) \,\varphi
        \,u(s)\,dx ds =\int_\Omega u_0\,\varphi\,dx+\int^t_{0} \langle
        f(s), \varphi \rangle\, ds\quad \forall \,\varphi\in
        V_2. \label{eq:weak-solution-3}
      \end{aligned}
    \end{gather}
  \end{subequations}
\end{definition}
Due to the fact that $\nu_0>0$, the existence of weak solutions
follows for all $p\geq1$ in a standard way, and one has not to resort
to very sophisticated tools as in Diening, R{\r u}{\v{z}}i{\v{c}}ka,
and Wolf~\cite{DRW2010}. We will come back later on, for the
motivation on this assumption on $\nu_0$. In particular, we do not
have any further restriction on $p$ and the proof follows the same
lines of the classical work on monotone operators, as summarized in
Lions~\cite{Lio1969}. The result below is part of the folklore
associated with non-Newtonian fluids. We will give a sketch of the
proof since some of the calculations will be used many times in the
sequel.
\begin{theorem}
  \label{thm:existence-uniqueness-weak-solutions} 
  Let be given $\nu_0,\nu_1>0$, $p\in[1,2]$, $u_0\in L^2(\Omega)$ with
  $\div u_0=0$ and $(u_0\cdot n)_{|\Gamma}=0$, and $f \in
  L^2(0,T,V'_2)$. Then,  there exists a unique solution $u$
  to~\eqref{eq:genNewModel-m}-\eqref{eq:boundary-cond}
  satisfying~\eqref{eq:weak-solution-1}-\eqref{eq:weak-solution-3}.
  Moreover, the following estimates are verified
  \begin{equation*}
    \begin{aligned}
      \|u\|^2_{L^{\infty}(0,T; L^2)} + \nu_0 \|\nabla
      u\|^2_{L^2(0,T; L^2)} &\leq C
      % |u_0\|_2^2 + \frac{1}{\nu_0}\|f\|_{L^2(0, T;V'_2)}^2+
      % C(p)T\nu_1\delta^p,
      % \label{eq:time-int-test-against-u} 
      \\
      % \intertext{and} 
      \|u_t \|^2_{L^2(0, T; V'_2)} &\leq
      C,%\big( 1+ \|u \|_{ L^2(0,T; V_2)}\big) + \|f\|_{        L^2(0,T;V'_2)}.
      % \label{eq:time-int-test-against-ut}
    \end{aligned}
  \end{equation*}
  where $C= C(p, \delta, \nu_0,\nu_1, \|f\|_{L^2(0, T;V'_2)}, \|u_0\|_2, T, \Omega)$. 
\end{theorem}
\begin{proof}
  We deduce the \textit{a priori} estimates on which the existence of
  weak solutions to~\eqref{eq:genNewModel-m}-\eqref{eq:boundary-cond}
  is based.  More properly, one should consider approximate Galerkin
  solutions defined as follows. Let $\{\omega^r\}$, with $r \in \N$,
  be the eigenfunctions of the Stokes operator and let $\{\lambda^r\}$
  be the corresponding eigenvalues; we define
  $X_m:=\textrm{span}\{\omega^1,\ldots,\omega^m\}$ and $P_m$ is the
  orthogonal projection operator over $X_m$.  We will seek approximate
  functions $\bbun(t, x) =\sum^m_{r=1} c^m_r(t)\omega^r(x)$ as
  solutions of the system of equations below, for all $1 \leq r \leq
  m$, $t \in [0,T]$
\begin{equation*} 
%  \label{eq:weak-GNS-box}
  % \begin{gather}
  \begin{aligned}
    \int_\Omega\Big[\bbun_t\omega^r +\nu_0 \nabla \bbun
    \nabla\omega^r + \nu_1 S (\dom\bbun)\, \dom \omega^r &+
    (\bbun\cdot \nabla)\bbun\, \omega^r\Big]\,dx \label{eq:weak-GNS} =
    \langle f, \omega^r\rangle,
    \\
    \bbun(0) = P^m u_0.
  \end{aligned} 
  % \end{gather}
\end{equation*}
Taking the $L^2$-product of~\eqref{eq:genNewMod1-m} with $\bbun$,
using suitable integrations by parts and Young inequality we get
  \begin{equation*}
    \begin{aligned}
      \frac{1}{2}\frac{d}{dt}\|\bbun\|_2^2 + \nu_0\|\nabla \bbun\|_2^2
      + \frac{\nu_1}{2} \int_{\Omega}(\delta +& |\dom
      \bbun|)^{p-2}|\dom \bbun|^2dx
      % &=\langle f, u\rangle
      \\
      &\qquad\leq \frac{\nu_0}{2} \|\nabla \bbun\|_2^2 +
      \frac{1}{2\nu_0}\|f\|_{V'_2}^2,
    \end{aligned}
  \end{equation*}
  Using~\eqref{eq:2.20} and integrating in time we arrive at the
  following inequality
  \begin{equation*}% \label{eq:new-utility-estimate}
    \begin{aligned}
      \|u^m(t)\|_2^2 + \nu_0\int_0^t\|\nabla u^m(s)\|_2^2\,ds
      &+C\nu_1\int_0^t\|\dom u^m(s)\|_p^p\,ds
      \\
      &\leq \|u_0\|_2^2 + \frac{1}{\nu_0} \int_0^t\|f (s)\|_{-1,
        2}^2\,ds+C(p)\nu_1\delta^p,
    \end{aligned}
  \end{equation*}
  for a.e. $t\in[0,T]$. We estimate, by comparison, the time
  derivative. The only term which requires some care is the
  extra-stress tensor $S$. %
%By the definition of the extra stress tensor
  %$S$, we get s
Since $p\leq2$ we get
  \begin{equation*}
    \begin{aligned}
      \int_0^T\langle S(\dom \bbun),\dom \varphi\rangle\,ds&\leq
      \|S(\dom\bbun)\|_{L^2(\Omega_T)}\|\nabla\varphi\|_{L^2(\Omega_T)}
      \\
      &\leq
      \|\dom\bbun\|_{L^{2p-2}(\Omega_T)}^{p-1}\|\nabla\varphi\|_{L^2(\Omega_T)}
      \\
      & \leq C(T,\Omega)\|\nabla\bbun\|_{L^2(\Omega_T)}^{p-1}\|\nabla\varphi\|_{L^2(\Omega_T)}.
    \end{aligned}
  \end{equation*}
  Whence, by standard
  calculations
  \begin{equation}
    \label{eq:existence-time-derivative} 
    \int_0^t\| \bbun_t(s) \|_{-1,2}^2\,ds \leq C,
  \end{equation}
  for a constant $C$ depending on $p$, $\nu_0$, $\nu_1$,
  $\|f\|_{L^2(0,T;V'_2)}$, $\|u_0\|_2$, $T$, and $\Omega$.  This
  proves that if $u^m$ is a Galerkin approximate solution then,
  uniformly in $m\in\N$,
\begin{equation*}
  u^m\in L^\infty(0,T;L^2(\Omega))\cap L^2(0,T;V_2)
  \quad\text{and}\quad        u^m_t\in L^2(0,T;V_2').
\end{equation*}
Note that we can extract sub-sequences converging weakly to some $u$
in $L^2(0,T;V_2)$, weakly* in $L^\infty(0,T;L^2(\Omega))$ and, by
Aubin-Lions theorem, strongly in $L^2(\Omega_T)$, and a.e.~in
$\Omega_T$. We have enough regularity to pass to the limit in the
convective term. Moreover, since $S(\dom u^m)$ is bounded uniformly in
$L^2(\Omega_T)$, it follows that $S(\dom u^m)\rightharpoonup A$ for
some $A$ in $L^{2}(\Omega_T)$. (Observe that without the Laplacian
term we would have only a bound in $L^{p'}(\Omega_T)$). We have now to
check that $A=S(\dom u)$.  This is obtained with the monotonicity
trick, see e.g.~\cite[\S2-5.2]{Lio1969}. By usual Sobolev embeddings
(since we are in two dimensions) the function $t\mapsto\int_\Omega
(u\cdot\nabla)\, u\,u\,dx\in L^1(0,T)$, hence we can write the energy
equality between any couple $0\leq s_0\leq s\leq T$
\begin{equation}
  \label{eq:energy-equality}
  \frac{1}{2}\|u(s)\|^2_2+\nu_0\int_{s_0}^s\|\nabla
  u\|^2_2 \,d\tau+\nu_1\int_{s_0}^s\langle A,u\rangle\,d\tau
  =\frac{1}{2}\|u(s_0)\|^2_2+\int_{s_0}^s(f,u)\,d\tau.
\end{equation}
Defining for $\phi\in L^2(0,T;V_2)$ (a test function with the same
regularity of $u$)
\begin{equation*}
  \mathcal{X}^m_s:=\nu_1\int_0^s\langle S(\dom\bbun)-S(\dom
  \phi),\dom\bbun-\dom\phi\rangle\,d\tau+
  \nu_0\int_0^s\|\nabla u^m\|^2_2\,d\tau+\frac{1}{2}\|u^m(s)\|^2_2,
\end{equation*}
it follows, by using that $S$ is monotone and by semi-continuity of
the norm, that
$\liminf_m \mathcal{X}^m_s\geq\nu_0\int_0^s\|\nabla
u\|^2_2\,d\tau+\frac{1}{2}\|u(s)\|^2_2$, and also that
\begin{equation*}
  \lim_{m} \mathcal{X}^m_s=\int_0^s (f,u)+\frac{1}{2}\|u_0\|^2_2
  -\nu_1\int_0^s\langle
  A,\dom\phi\rangle\,d\tau-\nu_1\int_0^s\langle S(\dom
  \phi),\dom u-\dom \phi\rangle\,d\tau.
 \end{equation*}
Hence,  by using the equality~\eqref{eq:energy-equality} we get
\begin{equation*}
  \nu_1\int_0^s\langle A-S(\dom \phi),\dom u-\dom
  \phi\rangle\,d\tau\geq0
  \qquad a.e.\ s\in [0,T].
\end{equation*}
We fix $\phi=u-\lambda\, \psi$ for $\psi\in L^2(0,T;V_2)$ and
$\lambda>0$. Finally, 
letting $\lambda\to 0^+$ the thesis follows.

It is important to point out that the weak solution above constructed 
is unique. Let us suppose that we have two solutions $u_1$ and $u_2$
corresponding to the same data. We obtain the following inequality for
$U:=u_1-u_2$ (This follows by using the usual interpolation
inequalities as for the Navier-Stokes equations and since $U$ is
allowed as test function, see Constantin and Foias~\cite{CF1988})
\begin{equation*}
  \begin{aligned}
    \|U(t)\|^2_2+\nu_0\int_0^t\|\nabla U(s)\|^2_2\,ds+\nu_1\int_0^t\langle
    S(\dom u_1)-S(\dom u_2),\dom u_1-\dom u_2\rangle\,ds
    \\
    \leq \frac{C}{\nu_0}\int_0^t\|\nabla u_1(s)\|^2_2 \|U(s)\|^2_2\,ds.
  \end{aligned}
\end{equation*}
Since $S$ is monotone (cf.~Lemma~\ref{eq:lemma-S}) the integral
involving the extra stress-tensor is non-negative. Using the Gronwall
lemma and the energy estimate one obtains that $U\equiv0$.
\end{proof}  

This latter result is very relevant since it allows to conclude that
\textit{all} the sequence $\{u^m\}$ converges to $u$. Moreover, if we
have other \textit{a priori} estimates on $u^m$, the extra-regularity
is inherited by weak solutions directly. This will be used in the
proof of Theorem~\ref{prop:unicita-caso-periodico}. Observe also that,
at moment, we do not have any information on the pressure, apart that
there exists as a distribution, by using De Rham theorem.

\section{The space-periodic case} 
\label{periodic-case-section} 
In this section we are concerned with the space-periodic case, that is
$\Omega = \T^2$. Each considered function $w$ will satisfy $w(x+2\pi
e_i) = w(x)$, $i = 1, 2$, where $\{e_1, e_2\}$ is the canonical basis
of $\R^2$.  We also require all functions to have vanishing mean
value, to ensure the validity of the Poincar\'e inequality.  We prove
some regularity results and we will show why the hypothesis $\nu_0>0$
seems necessary in many arguments. We define
$\mathcal{V}_{\textrm{per}}(\Omega)$ as the space of vector-valued
functions on $\Omega$ that are smooth, divergence-free, and space
periodic with zero mean value. For $1<q<\infty$ and $k \in \N$, set
\begin{equation*}
   W^{k,q}_{\div}(\Omega) := \big\{ \textrm{closure of}\quad
  \mathcal{V}_{\textrm{per}}(\Omega)\quad \textrm{in}\quad
  W^{k,q}(\Omega) \big\},
\end{equation*}
endowed, with the usual norms.

In the space periodic setting many calculations are simpler since we
can use $-\Delta u$ as test function (now formally but the procedure
goes through the Galerkin approximation).  Since  in the
2D~space-periodic case $\int_\Omega (u\cdot\nabla)\,u \Delta u\,dx=0$ we get
\begin{equation}
  \label{eq:easy}
\frac{d}{dt}\|\nabla u\|^2+\nu_0\|\Delta u\|^2+\nu_1\mathcal{I}(u)\leq C\|f\|^2,
\end{equation}
hence, if we are able to construct such a solution (this is not
trivial at all due to some technical issues when passing to the limit
in $\int_0^T\mathcal{I}(\bbun(s))\,ds$, for a fixed $T>0$) that and
if $\nu_0=0$ we obtain as higher order estimate
\begin{equation*}
  \int_0^T\mathcal{I}(u)\,dt=  \int_0^T\int_{\T^2}(\delta+|\dom
  u|)^{p-2}|\nabla\dom u|^2\,dx dt<+\infty. 
\end{equation*}
We recall the following lemma, which is an adaption
of~\cite[Lemma~4.4]{BE-DI-RU:2010} to the two dimensional case.
\begin{lemma} 
  \label{lem:lemma-10} 
  Let $p\in (1, 2]$, $\delta\in(0,\infty)$, and $\ell\in [1,2)$.  Then,
  for all sufficiently smooth functions $u$ with vanishing mean value
  over $\Omega$, the following relations hold true
  \begin{align*}
    &\| u \|^p_{2, \ell}\leq c \big(\mathcal{I}(u) +
    \delta^p\big),
    \label{lem:lemma-10-1}
  \end{align*}
\end{lemma}
Hence, the information on the regularity in the space variable which
we can extract from~\eqref{eq:easy}, in the case $\nu_0=0$, could be at most
\begin{equation*}
  u\in W^{2,\ell}(\T^2)\qquad \forall\, \ell<2,\quad a.e. \ t\in[0,T].
\end{equation*}
This is not enough to employ Thm.~\ref{thm:Dashti-Robinson} and
explains the introduction of the hypothesis $\nu_0>0$.

\begin{proof}[Proof of   Proposition~\ref{prop:unicita-caso-periodico}]
  In the light of the above observations the proof follows as in the 2D
  Navier-Stokes equations, see\cite{DA-RO:2009}. We test the equations
  by $-t\,\Delta u^m$ and we have
\begin{equation*}
  \frac{d}{dt}(t\,\|\nabla u^m\|^2)+\nu_0\,t\,\|\Delta
  u^m\|^2+\nu_1\,t\,\mathcal{I}(u^m)\leq C\,t\,\|f\|^2+\|\nabla u\|^2.
\end{equation*}
Hence, no matter of the non-negative term coming from the extra-stress
tensor, integrating in time over $[0,T]$ we have that
$\sqrt{t}\,u^m\in L^2(0,T;W^{2,2}(\T^2))$.  Due to uniqueness of the
solution the whole sequence $\{u^m\}$ converges to $u$ and by
lower-semicontinuity of the norm we obtain that $\sqrt{t}\,u\in
L^2(0,T;W^{2,2}(\T^2))$.
\end{proof}
For the sake of completeness, we recall that in the periodic 2D case,
with $\nu_0=0$ it is possible to prove the following result of
existence of regular solutions, see Kost~\cite{Kos2010}, which is an
adaption of those in~\cite{BE-DI-RU:2010} for the 3D case.  (Observe
that in absence of the Laplacian also the existence and uniqueness of
weak solutions is more delicate and the limit process on Galerkin
solutions requires some care). The following result, which is of
interest by itself, is not enough for our purposes of studying
uniqueness for solutions to~\eqref{eq:traj}.
\begin{proposition} 
  \label{eq:prop-nuzero} 
  Let be given $\delta\in[0,\delta_0]$, for some $\delta_0>0$, set
  $\nu_0=0$, $\nu_1>0$, and let $p\in (1,2]$.  Given $T>0$, assume
  that $f\in L^\infty(0,T;W^{1,2}(\T^2))\cap W^{1,2}(0, T;
  L^2(\T^2))$. Let $u_0\in W^{2,2}(\T^2)$ be such that $\div u_0=0$
  and $\div S(\dom u_0)\in L^2(\T^2)$.  Then, there is a time
  $0<T'\leq T$ (depending on the data of the problem) such that the
  system~\eqref{eq:genNewModel-m}, has a strong solution $u$ on $[0,
  T']$ satisfying, for $r\in(4/3,2)$,
  \begin{equation*}
    u \in L ^q (0,T';W^{2,r}(\T^2))\cap C (0,T';W^{1,q}(\T^2)),\qquad
    \forall\,q<\infty.   
  \end{equation*}
\end{proposition}
\begin{remark}
  One can obtain further regularity results for $u_t$ and also for
  $\nabla\pi$ (the latter if $\delta>0$).
\end{remark}
\section{Space-time regularity in the Dirichlet case}
\label{sec:time-regularity}
In this section we consider the time evolution problem with Dirichlet
boundary conditions and we prove a result of regularity for smooth
data. Then, we will relax some of the assumptions to prove the main
result of the paper.  We start by showing a first regularity result
for the time derivative of the solutions to the
problem~\eqref{eq:genNewModel-m} with Dirichlet boundary
conditions. We prove now some results by using as test functions first
and second order time derivatives of the velocity. These are legal
test functions, since if $u$ is divergence-free and $u_{|\Gamma}=0$,
then $\frac{\partial^k u}{\partial t^k}$ shares the same two
properties, for all $k\in\N$. In particular, the following result is
valid in any smooth and bounded domain, while the hypothesis of flat
boundary will be used for the $W^{2,2}(\Omega)$-regularity.
\begin{lemma} 
  \label{lem:lemma-C} 
  Let $p \in (1, 2]$, $\delta>0$, $f\in W^{1,2}(0,T;L^2(\Omega))$,
  $u_0\in W^{2,2}(\Omega)\cap V_2$, and let $u$ be a weak solution of
  problems~\eqref{eq:genNewModel-m}-\eqref{eq:boundary-cond}. Then,
  \begin{equation} \label{eq:time-time}
    \begin{aligned}
      \|u_t\|^2_{L^\infty(0, T;L^2)}+ \|\nabla u
      \|^2_{L^\infty(0, T;L^2)} &+ \nu_0\|\nabla u_t\|^2_{L^2(0, T;L^2)}
      \\
      &
      +\nu_1\|\mathcal{J}(u)\|_{L^1(0, T)} \leq C,
    \end{aligned}
  \end{equation}
  where the constant $C$ depends on $p$, $\delta$, $\nu_0$, $\nu_1$,
  $\|f\|_{W^{1,2}(0, T;L^2)}$, $\|u_0\|_{2,2}$, $T$, and
  $\Omega$.
\end{lemma}
As in the previous result we only prove the a priori estimates.  A
complete proof can be obtained through a Galerkin approximation and
for the reminder of this section we drop the superscript $``m"$. We
also define    
\begin{equation*}
     M(t):= \int_0^t \big(\delta + s)^{p-2}s \, ds \geq 0, \quad 
     \textrm{for}\quad  t\geq 0.
  \end{equation*}
  Observe that $M(t) \simeq (\delta + t)^{p-2}t^2$ and also
  $(\delta+t)^{p-2}t^2\leq t^p$, with $1\leq p\leq2$. This shows that
  \begin{equation} 
    \label{eq:stress-tens-relation} 
0\leq    \mathcal{M}(u):=\int_\Omega M(|\dom u |) \,dx\leq C(p)\|\dom u\|_p^p,\quad \textrm{with}\quad
    1\leq p\leq 2. 
  \end{equation}
\begin{proof}[Proof of Lemma~\ref{lem:lemma-C}]
  First, we multiply~\eqref{eq:genNewMod1-m} by $u_{t}$ and integrate
  by parts.  We observe that taking the duality of $-\div S(\dom u)$
  against $u_t$, we get
  \begin{equation} 
    \label{eq:idea-furba} 
    -\big\langle{\div}\,
    S(\dom u), u_t\big\rangle = \big\langle S(\dom u ), \dom
    u_t\big\rangle = \frac{d}{dt}\mathcal{M}(u).
  \end{equation}
  By suitable integrations (since $\div u_t=0$) we obtain
  \begin{equation*}
    \|u_t\|_2^2 + \frac{\nu_0}{2}\frac{d}{dt} \|\nabla
    u\|_2^2+\nu_1\frac{d}{dt}\mathcal{M}(u)=
    \int_\Omega  \big(f\,  u_t-(u\cdot\nabla)\,u\, u_t\big) \,dx. 
  \end{equation*}
  By using H\"older and Gagliardo-Nirenberg inequalities, with the
  boundedness of the kinetic energy, we get, for all $\varepsilon>0$
  \begin{equation*}
    \begin{aligned}
      \left| \int_\Omega (u\cdot\nabla)\,u\, u_t\,dx \right| &\leq
      \|u\|_4\|\nabla u\|_2\|u_t\|_4
      \\
      &\leq C\|u\|_2^{\frac{1}{2}}\|\nabla u\|_2^{\frac{1}{2}}
      \|\nabla u\|_2\|u_t\|_2^{\frac{1}{2}}\|\nabla
      u_t\|_2^{\frac{1}{2}}
      \\
      &\leq c_\varepsilon \big(\|\nabla u\|_2^2 +\|u_t\|_2^2 \|\nabla
      u\|_2^2\big) + \varepsilon \|\nabla u_t\|_2^2,
    \end{aligned}
  \end{equation*}
  Thus, we obtain the following differential inequality
  \begin{equation}
    \label{eq:timed1}    
    \|u_t \|_2^2 + \frac{d}{dt} \big({\nu_0}\|\nabla
    u\|_2^2+{\nu_1}\mathcal{M}(u)\big)
    \leq  c_\varepsilon \big(\|\nabla u\|_2^2 +\|u_t\|_2^2 \|\nabla
    u\|_2^2+\|f\|_{2}^2\big) + \varepsilon \|\nabla u_t\|_2^2,
  \end{equation}
  which  we clearly cannot use directly, due to the lack of
  control for $\nabla u_t$.
\begin{remark}
  Another path will be that of using improved estimates for $\nabla u$
  to estimate the convective term, see the last section.
\end{remark}
We take now the time derivative of~\eqref{eq:genNewMod1-m}, multiply
by $u_t$ and integrate by parts (recalling that
$\int_\Omega(u\cdot\nabla)\,u_t\,u_t\,dx =0$) to obtain
  \begin{equation} 
    \label{eq:first-result-intermediate-ineq}
    \frac{1}{2}\frac{d}{dt} \|u_t\|^2_2 + \nu_0 \|\nabla u_t\|^2 +
    \nu_1\langle \D_t\big(S(\dom u)\big), \dom u_t\rangle \leq
    \left|\int_\Omega \Big((u_t\cdot\nabla)\,u\, u_t+ f_t\, u_t\Big)\,dx\right|.
  \end{equation}
  By~\eqref{eq:ass-1a} the term involving $S$
  in~\eqref{eq:first-result-intermediate-ineq} is non-negative being
  estimated from below by a multiple of $\mathcal{J}(u)\geq0$.  Let us
  focus on the right-hand side
  of~\eqref{eq:first-result-intermediate-ineq}. By using  H\"older
  and interpolation inequality, and the energy estimate we get, for
  each $\eta>0$,
  \begin{equation*}
    \begin{aligned}
      \left|\int_\Omega(u_t\cdot\nabla)\,u\,u_t\,dx\right|
      =\left|\int_\Omega(u_t\cdot\nabla)\,u_t\,u\,dx\right|
      &\leq \|u_t\|_4\|\nabla u_t\|_2\| u\|_4
      \\
      &\leq C\|u_t\|_2^{\frac{1}{2}} \|\nabla u_t\|_{2}^{\frac{3}{2}}
      \| u\|_2^{\frac{1}{2}} \| \nabla u\|_{2}^{\frac{1}{2}}
      \\
      &\leq c_\eta\|\nabla u\|_{2}^2\|u_t\|_{2}^2 
      +\eta\|\nabla u_t\|_{2}^2,
    \end{aligned}
  \end{equation*}
  hence, choosing $\eta>0$ small enough we get
  \begin{equation}
    \label{eq:timed2}
    \frac{d}{dt}\|u_t\|^2_2 + \nu_0 \|\nabla u_t\|^2_2
    + \nu_1\mathcal{J}(u) 
    \leq    C\big( \|\nabla u\|_{2}^2\|u_t\|_2^2+ \|f_t\|_{2}^2\big).
  \end{equation}
  Summing up~\eqref{eq:timed1}-\eqref{eq:timed2} and choosing
  $\varepsilon>0$ small enough we get finally
  \begin{equation*}
    \begin{aligned}
      \frac{d}{dt}\Big(\|u_t\|_2^2+\nu_0\|\nabla
      u\|_2^2+\nu_1\mathcal{M}(u)\Big)+\|u_t\|_2^2+\nu_0\|\nabla
      u_t\|_2^2+\nu_1\mathcal{J}(u)
      \\
      \leq C \Big(\|\nabla u\|_2^2\|u_t\|_2^2 +\|\nabla
      u\|_2^2+\|f\|_{2}^2+\|f_t\|_{2}^2\Big).
    \end{aligned}
  \end{equation*}
  To integrate over $[0,T]$ we need to make sense to
  $\|u_t(0,\cdot)\|_2$. From the assumptions on the data, the fact that
  $\delta>0$, and $u^m(0) = P^mu_0$ we easily get
  (cf.~\cite[\S~5]{BE-DI-RU:2010}) that
  \begin{equation*}
  \|u_t^m(0)\|_2\leq c\big( \|u^m_0\|_{2,2}^2 +\|f^m(0)\|_2\big).
    \end{equation*}
    Recall that we are working on the finite dimensional approximation
    $u^m$ and taking the limit $m\to+\infty$. With Gronwall lemma and
    by using the fact that $\nabla u\in L^2(0,T;L^2(\Omega))$, we get
    for a.e $t\in [0, T]$
  \begin{equation*}
    \begin{aligned}
      \| u_t (t) \|^2_2 + \nu_0 \|\nabla\ut\|_2^2+
%      \mathcal{M}(u(t))    +
   &   \int_0^t \Big(\|u_t(s)\|_2^2+ {\nu_0}\|\nabla u_t (s)\|^2_2
      + \mathcal{J}(u(s))\big)\, ds
      \\
      &\qquad \leq C (\nu_0,\nu_1,\delta,T, \|f\|_{W^{1,2}(0,T;L^2)},\|u_0\|_{2,2},\Omega),
    \end{aligned}
  \end{equation*}
 hence the thesis.
\end{proof}
\begin{remark}
  The hypotheses on the external force can be slightly relaxed, but
  this is inessential in our treatment. 
\end{remark}
We now prove Proposition~\ref{lem:main-theo-with-delta-with-nonlin}. For the
reader's convenience we split the proof into two parts. First, we
perform a preliminary study of the system obtained removing the
convective term $(u\cdot\nabla)\,u$ from~\eqref{eq:genNewMod1-m}.
\begin{subequations} 
  \label{eq:without-nonlin}
  \begin{align}
    u_t -\nu_0\Delta u -\nu_1 \div\, S (\dom u) + \nabla \pi = f\quad \,
    \textrm{in}\quad \,& [0, T] \X \Omega,
    \label{eq:genNewMod1-nonl}
    \\
    \div  u = 0 \quad \,\textrm{in}\quad \, & [0, T] \X \Omega,
    \label{eq:genNewMod2-nonl} 
    \\
    u = 0 \quad \,\textrm{in}\quad \, & [0, T] \X \Gamma,
    \\
    u (0) = u_0 \quad \, \textrm{in}\quad \, &\Omega,
    \label{eq:genNewMode3-nonl}
  \end{align}
\end{subequations}
and focusing on the role of the nonlinear stress-tensor.  The
system~\eqref{eq:without-nonlin} can be treated similarly to a steady
state problem if we have good enough \textit{a priori} estimates on
$u_t$. We will then address the full
problem~\eqref{eq:genNewModel-m}-\eqref{eq:boundary-cond}, by adding
suitable estimates for the convective term.
\begin{lemma} 
  \label{lem:main-theo-with-delta} 
  Let $\nu_0>0$, $\delta>0$ and $p\in \big[\frac{3}{2}, 2\big]$. Given
  $T>0$, assume that $u_0\in W^{2,2}(\Omega)\cap V_2$ and $f\in
  W^{1,2}(0, T; L^2(\Omega))$. Then,
  problem~\eqref{eq:without-nonlin}-\eqref{eq:boundary-cond} admits a
  unique solution, such that~\eqref{eq:regularity-maximal} holds true.
\end{lemma}
\begin{proof}%[Proof of Lemma~\ref{lem:main-theo-with-delta}] 
  We adapt to the time-dependent case a technique with three
  intermediate steps taken from~\cite{BEI:2009,CR:2009}: In the first
  step we bound the tangential derivative of velocity and pressure; In
  the second step we estimate the normal derivative of the velocity
  field; In the last step we estimate the normal derivative of the
  pressure.

  Again we merely prove the \textit{a priori} estimates. Observe that
  for this simpler problem without convection, the same existence
  proved in Theorems~\ref{thm:existence-uniqueness-weak-solutions} and
  regularity from Lemma~\ref{lem:lemma-C} clearly hold true
  (this is particularly relevant for what concerns $u_t$).

\bigskip

\noindent{\textbf{Step 1.}} We first prove that the following
estimates, concerning the tangential derivatives, hold true
      \begin{equation}
        \label{eq:solution-reg-1}
        \nu_0\|\nabla \D_1
        u\|^2+\nu_0\|\D^2_{22}u_2\|^2_{L^2(0,T;L^2)}
        +\|\D_1\pi\|_{L^2(0,T;L^2)}\leq C,
    \end{equation}
    where $C$ depends on $p$, $\delta$, $\nu_0$, $\nu_1$,
  $\|f\|_{W^{1,2}(0, T;L^2)}$, $\|u_0\|_{2,2}$, $T$, and
  $\Omega$.

\smallskip

We now use the particular features of the flat domain.  Multiplying
equation~\eqref{eq:genNewMod1-nonl} by $-\D_{11}^2 u$ and integrating 
by parts, it follows that
  \begin{equation*}
    \frac{1}{2}\frac{d}{dt} \|\D_1u\|^2_2
    +\nu_0\|\nabla \D_1 u\|^2_2 
    +(p-1)\nu_1\!\int_\Omega  (\delta + |\dom u|)^{p-2} |\D_1 \mathcal{D} u|^2\,dx
    \leq \|f\|_2\,\|\D_{11}^2 u\|_2.
  \end{equation*}
  By applying Young inequality and using relation~\eqref{eq:I1}, we
  get a.e. in $[0,T]$
  \begin{equation}
    \label{eq:estimate-test-11-final} 
    \begin{aligned}
      \|\D_1u(t)\|^2_2 +\int_0^t\big(\nu_0\|\nabla \D_1 u(s)\|^2_2
      &+%(p-1)
      \nu_1\mathcal{I}_1(u(s))\big)\,ds 
      \\
      &\leq C\big(\|\nabla
      u_0\|^2+\frac{1}{\nu_0}\int_0^t\|f(s)\|_2^2\,ds\Big),
    \end{aligned}
  \end{equation}
  and, since $\div u=0$, $\D^2_{22}u_2=-\D^2_{21} u_1$ and the
  estimate on $\D^2_{22} u_2$ follows.
 
  Let us focus on the pressure term. Differentiating the
  equation~\eqref{eq:genNewMod1-nonl} with respect to the tangential
  direction $x_1$, one has that
  \begin{equation*}
    \label{eq:pressure-estimate}
    \nabla\D_1\pi = \nu_0 \div \D_1 \nabla u + \nu_1\div
    \D_1\big[(\delta + |\dom u|)^{p-2}  \mathcal{D} u\big] -\D_1u_t +
    \D_1 f\qquad a.e. \text{ in }\Omega_T.
  \end{equation*}
 We observe that $\partial_1
  u_t=\div \left(\begin{matrix}
    \partial_t u_1&0
    \\
    \partial_t u_2&0
  \end{matrix}
\right)$ and $\partial_1 f=\div
\left(\begin{matrix}
    f_1&0
    \\
    f_2&0
  \end{matrix}
\right)$. Hence to apply Lemma~\ref{lem:lemma-2.6-crispo} to estimate
$\partial_1 \pi$, we only have to bound the term
$\D_1\big[(\delta+|\dom u|)^{p-2} \mathcal{D} u\big]$. A direct
computation gives
  \begin{equation*}
    \begin{aligned}
      \D_1\big[(\delta +  |\dom u|)^{p-2}  \mathcal{D} u\big]
      = (\delta + |\dom u|)^{p-2} \D_1 \mathcal{D} u + (p-2) (\delta
      + |\dom u|)^{p-3}  (\mathcal{D} u\cdot \D_1
      \mathcal{D} u)\frac{\mathcal{D} u}{|\mathcal{D}u|},
    \end{aligned}
  \end{equation*}
  and consequently
  \begin{equation*}
    \left| \D_1\big[(\delta +  |\dom u|)^{p-2}  \mathcal{D} u\big] \right|
    \leq (3-p)  (\delta + |\dom u|)^{p-2} | \D_1 \mathcal{D} u |\qquad
    a.e. \text{ in }\Omega_T.
  \end{equation*}
 Therefore, by comparison $\D_1\big[(\delta + |\dom u|)^{p-2}
  \mathcal{D} u\big]\in L^2(\Omega)$ and it follows that
  \begin{equation*}
    \int_\Omega \big| \D_1\big[(\delta +  |\dom u|)^{p-2}  
    \mathcal{D} u\big]\big\vert^2\,dx
    \leq c\,\delta^{p-2}\,\mathcal{I}_1(u)\qquad a.e.\ t\in [0,T].
  \end{equation*}
  By applying Lemma~\ref{lem:lemma-2.6-crispo} we have that
  \begin{equation*}
    \|\D_1\pi\|_2^2\leq  \|u_t\|_2^2 + \nu_0\|\D_1 \nabla u\|_2^2 + 
    \nu_1\,C\,\delta^{p-2}\,\mathcal{I}_1(u) + \|f\|_2^2 \qquad a.e.\ t\in[0,T],    
  \end{equation*}
  from which, integrating in time over $[0,T]$,
  using~\eqref{eq:estimate-test-11-final} and recalling the bounds
  previously proved on $u_t,\,\partial_1\nabla u$, and
  $\mathcal{I}_1$, then~\eqref{eq:solution-reg-1} follows.

  \bigskip

  \noindent{\textbf{Step 2.}} To bound
  $\|\D_{22}^2u_1\|_{L^{2}(0,T;L^2)}$, we consider a narrower range of
  values for the parameter $p$. Under the same hypotheses as before,
  but for $p\in\big[\frac{3}{2}, 2\big)$, we have
  \begin{equation*}
%    \label{eq:estimate-d22-u1}
    \|\D^2_{22}u_1\|_{L^2(0, T;L^2)}\leq C,
  \end{equation*}
  where the constant $C$ depends on $p$, $\delta$, $\nu_0$, $\nu_1$,
  $\|f\|_{W^{1,2}(0, T;L^2)}$, $\|u_0\|_{2,2}$, $T$, and $\Omega$.

\medskip

We follow the main lines established in the proof
of~\cite[Lemma~3.3]{CR:2009}.  By calculating $\D_2\big[(\delta+|\dom
u|)^{p-2} \mathcal{D} u\big]$, the first equation
in~\eqref{eq:genNewMod1-nonl} can be written as 
\begin{equation}
    \label{eq:estimate-double-der-u1}
      \alpha_1\D^2_{22} u_1 = -F_1 -f_1+ \D_t u_1+ \D_1\pi,
  \end{equation}
  where
  \begin{equation*}
    \begin{aligned}
      \alpha_1&:=\nu_0 + \frac{\nu_1}{2} (\delta+|\dom u|)^{p-2} +
      \nu_1 (p -2) \frac{(\delta+|\dom u|)^{p-3}}{|\dom u|} (\dom
      u)_{12}(\dom u)_{12},
      \\
&\text{and}
      \\
      F_1&:= \big[\nu_0 +\nu_1(\delta + |\dom
      u|)^{p-2}\big]\D_{11}^2u_1+\frac{\nu_1}{2}(\delta + |\dom
      u|)^{p-2}\D^2_{12} u_2
      \\
      &\quad+\nu_1(p-2)\frac{(\delta + |\dom u|)^{p-3}}{|\dom
        u|}\Big[ \sum_{k,l=1}^2(\dom u)_{k l}\D_1 (\dom u)_{k l}(\dom u)_{11} +
      \D^2_{22}u_2 (\dom u)_{22}(\dom u)_{12}
      \\
      &\hspace{8cm}+\frac{\D^2_{12}u_2}{2} (\dom u)_{12}(\dom u)_{12}
      \Big].
 \end{aligned}
  \end{equation*}
  By direct calculations it can be easily seen that
  \begin{equation*}
    | F_1|\leq C\Big[\nu_0 + \nu_1\Big(p-\frac{3}{2}\Big)
    (\delta+|\dom u|)^{p-2}\Big]|D^2_\ast u|\qquad a.e.\text{ in }\Omega_T
  \end{equation*}
  and by using that $p\geq\frac{3}{2}$ we get
  \begin{equation*}
%    \label{eq:coeff-estimate}
    \alpha_1\geq \Big[\nu_0 + \nu_1\Big(p-\frac{3}{2}\Big)
    (\delta+|\dom u|)^{p-2}\Big]\geq \nu_0>0.
  \end{equation*}
  Division of  both sides of~\eqref{eq:estimate-double-der-u1} by
  $\alpha_1$ is then legitimate and we infer that
  \begin{equation*}
    |\D^2_{22}u_1|
    \leq C\big( |D_\ast^2 u| + 
    \frac{1}{\nu_0}\big(|\D_1\pi| + |\D_t u_1| + |f_1|\big)\qquad a.e.\text{ in }\Omega_T.
  \end{equation*}
  Therefore, squaring and integrating over $\Omega_T$ we get
  \begin{equation*}
    \int_0^T \|\D^2_{22}u_1(s)\|_2^2\,ds
    \leq     \frac{C}{\nu_0}\int_0^T\Big( \|D_\ast^2 u(s)\|^2_2 + 
\|\D_1\pi(s)\|^2_2 + \|\D_t u_1(s)\|^2_2 + \|f_1(s)\|^2_2\Big)\,ds,
  \end{equation*}
which, by %Lemma~\ref{lem:estimate-for-the-tangential-derivatives} and
the previous results  is finite.
This finally shows that $D^2 u\in L^2(\Omega_T)$.

\bigskip

\noindent{\textbf{Step 3.}} The final step, which is not strictly
required for the particle trajectories uniqueness, is the regularity
of the normal derivative of pressure. Nevertheless, we include it for
the sake of completeness. Under the same hypotheses we have
  \begin{equation*}
    % \label{eq:estimate-D2-pi}
    \|\D_2\pi\|_{L^2(0, T;L^2)}\leq C,
  \end{equation*}
  where the constant $C$ depends on $p$, $\delta$, $\nu_0$, $\nu_1$,
  $\|f\|_{W^{1,2}(0, T;L^2)}$, $\|u_0\|_{2,2}$, $T$, and $\Omega$.

\medskip

By using the  second equation in~\eqref{eq:genNewMod1-nonl}, one  can
write
  \begin{equation*}
    \begin{aligned}
      |\D_2\pi| &\leq c\Big(\nu_0 + \nu_1(p-2)
      (\delta + |\dom u|)^{p-2}\Big)  |D^2 u| + |\D_2u_t| +
      |f_2|\qquad a.e.\text{ in }\Omega_T.
    \end{aligned}
  \end{equation*}
  Hence, straightforward calculations lead to
  \begin{equation*}
    \int_0^T\|\D_2\pi (s)\|_2^2\,ds 
    \leq c\int_0^T\bigg(\Big[ \nu_0 + \nu_1\delta^{2(p-2)} \Big]\|D^2 u (s)\|_2^2 
    +  \|\D_2u_t (s)\|_2^2 + \|f (s)\|_2^2\bigg)\, ds,
  \end{equation*}
  and the assertion follows as a consequence of the previous results.
\end{proof}
We finally prove the same regularity results also in presence of the
convective term. We use a perturbation argument, treating
$(u\cdot\nabla)\,u$ as a right-hand side in
equation~\eqref{eq:genNewMod1-m}.
\begin{proof}[Proof of
  Proposition~\ref{lem:main-theo-with-delta-with-nonlin}]
  Here, we use the a priori estimates obtained for the
  problem~\eqref{eq:without-nonlin} with external body force
  \begin{equation*}
    \mathcal{F}:=  -(u\cdot \nabla)\,u + f.
  \end{equation*}
  In the derivation of estimates for $u_t$ we used that
  $\|f\|_{W^{1,2}(0,T;L^2)}$, while in
  Lemma~\ref{lem:main-theo-with-delta} the estimates depend
  essentially on the $L^2(\Omega_T)$-norm of the external force.
  Hence, by using Lemma~\ref{lem:lemma-C} it is then sufficient to
  estimate $\|(u\cdot\nabla)\,u\|_{L^2(0,T;L^2)}$ in terms of second
  order derivatives of $u$, to follow the same calculations in
  Step~1--3 of the previous result.

  By applying H\"older, Gagliardo-Nirenberg, and Young inequalities
  and the energy estimate, we get for each $\varepsilon>0$
%  \begin{subequations}
  \begin{equation}
  \begin{aligned} 
        \label{eq:nonlin-stima-1} 
        \|(u\cdot\nabla)\,u\|_2 \leq \|u \|_{4}\|\nabla u\|_{4} &\leq
        c\| u \|_{2}^{\frac{1}{2}}\|\nabla u \|_{2}^{\frac{1}{2}}
        \|\nabla u\|_{2}^{\frac{1}{2}}\|D^2 u\|_{2}^{\frac{1}{2}}
        \\
        & \leq c_\varepsilon \|\nabla u\|_2^2 + \varepsilon\|D^2 u\|_{2}.
      \end{aligned}
    \end{equation}
    By using the same calculations as in the previous proposition and
    the \textit{a-priori} estimates in~\eqref{eq:time-time}
    --especially that $\nabla u\in L^\infty(0,T;L^2(\Omega))$-- we have
\begin{equation*}
  \begin{aligned}
&    \int_0^T\big(\|u\|^2_{2,2}+\|\pi\|_{1,2}^2\big)\,ds\leq
    C\int_0^T\big(\|f\|_2^2+\|u_t\|_2^2+\|(u\cdot\nabla)\, u\|_2^2\big)\,ds
    \\
  &\qquad  \leq C(p,
    \delta,\nu_0,\nu_1,\|f\|_{W^{1,2}(0,T;L^2)},\|u_0\|_{2,2},T,\Omega,\varepsilon)
    +\varepsilon\int_0^T\|D^2u\|_{2}^2\,ds,
  \end{aligned}
\end{equation*}
and, by choosing $\varepsilon>0$ small enough, we end the proof.
\end{proof}
As a consequence of the above result we have full $L^2$-space-time
regularity of the solution up to second-order space-derivatives, hence
the uniqueness of particle trajectories. The result is not optimal in
view of application to uniqueness of trajectories, in
the sense that some of the hypotheses can be slightly relaxed. For
instance $f_t\in L^2(\Omega_T)$ and $u_0\in W^{2,2}(\Omega)$ can be
removed (at the price of less regularity on $u_t$) by following a
slightly different path as we do in the next section.
\section{Proof of Theorem~\ref{thm:teoUniq}}
\label{sec:uniqueness}
In this section we finally address the problem of the uniqueness of
particle trajectories under ``minimal'' assumptions on the data.  We will
show how the previous regularity result, together with
Theorem~\ref{thm:Dashti-Robinson}, allow us to prove
Theorem~\ref{thm:teoUniq}.
\begin{proof}[Proof of Theorem~\ref{thm:teoUniq}] In the same way as
  in the proof of Lemma~\ref{lem:main-theo-with-delta}, we perform
  separately the \textit{a priori} estimates for the normal and
  tangential derivative of the time-weighted $\sqrt{t}\, u^m$ (which we call
  $\sqrt{t}\, u$). In particular, here we do not use a lot of regularity on
  $u_t$, but we have to face with a non-smooth $u_0$.  By adapting
  standard weighted estimates, we multiply the
  equation~\eqref{eq:genNewMod1-m} by $-t\,\D_{11}^2 u$. Integrating by
  parts, and with Young inequality we obtain
  \begin{equation*}
    \begin{aligned}
      \frac{1}{2}\frac{d}{dt} \big(t\,\|\D_1 u \|^2_2 \big) &+ \nu_0 \,
      t\,\|\nabla \D_1 u \|_2^2 + (p-1) \nu_1\,t\, \mathcal{I}_1( u )
      \\
      &\leq \frac{\nu_0}{2}t\, \|\D_{11}^2 u \|_2^2 + \frac{C}{\nu_0}
      t\,\big(\|( u \cdot \nabla ) u \|_2^2 + \|f\|_2^2\big) + \|\D_1 u
      \|^2_2.
    \end{aligned}
  \end{equation*}
  Integrating in time and using the energy estimate to bound
  $\int_0^t\|\D_1u\|^2\,ds$ it follows%for a.e. $t\in [0, T]$
  \begin{equation} 
    \label{eq:nabla+1}
    \begin{aligned}
      &t\,\|\D_1 \ut\|_2^2 + \nu_0\int_0^t s\,\|\nabla \D_1
      \us\|_{2}^2\,ds + \nu_1 \int_0^ts\,\mathcal{I}_1(\us) \,ds
      \\
      &\qquad \leq \frac{C}{\nu_0}\Big[\|u_0\|_2^2+\int_0^t s\,\big(
      \|(\us\cdot \nabla)\,\us\|_2^2 +\|f (s)\|^2_2\big)\,
      ds\Big]\quad\text{a.e. in }[0,T].
    \end{aligned}
  \end{equation}
  We take now the $L^2$-inner product of~\eqref{eq:genNewMod1-m} with
  $t\,u_{t}$. By suitable integrations by parts, and
  using~\eqref{eq:stress-tens-relation}-\eqref{eq:idea-furba} we reach
  \begin{equation*}
    \begin{aligned}
      t\,\| u_t \|_2^2 &+ \frac{\nu_0}{2}\frac{d}{dt} \big(t\,\|\nabla
      u\|_2^2\big) + {\nu_1} \frac{d}{dt}\big( t\, \mathcal{M}(u)\big)
      \\
      & \leq t \left( \left|\int_\Omega
          (u\cdot\nabla)u\,u_t\,dx\right| + \left|\int_\Omega f\,
          u_t\,dx\right|\right) + \nu_0\,\|\nabla u\|_2^2 +
      C\,\nu_1\,\mathcal{M}(u)
      \\
      & \leq \frac{t}{4}\big(\|(u\cdot\nabla)\,u\|_2^2 +
      \|f\|_2^2\big) + \frac{t}{2}\|u_t\|^2 + \nu_0\|\nabla u\|_2^2 +
      C\,\nu_1\|\dom u\|^p_p.
    \end{aligned}
  \end{equation*}
  Integrating this inequality in time, by appealing to the energy
  inequality and recalling that $\mathcal{M}(u)\geq 0$, it follows
  that for a.e. $ t\in [0, T]$
  \begin{equation} 
    \label{eq:time-space}
    \begin{aligned}
      \nu_0\, t\, \|\nabla \ut\|_2^2 &+ \int_0^t s\, \|u_t (s)
      \|_2^2\,ds
      \\
      & \leq C\bigg[\|u_0\|_2^2+\int_0^t s\,\big(\|(\us\cdot
      \nabla)\,\us\|_2^2 +\|f(s)\|_2^2\big)\, ds\bigg],
    \end{aligned}
  \end{equation}
  where $C$ depends on $p$, $\delta$, $\nu_0$, $\nu_1$, $T$, and
  $\Omega$.

  Let us now focus on the normal derivatives of $u$.  Arguing as in
  Step~2 of the proof of Lemma~\ref{lem:main-theo-with-delta}, and
  replacing $f$ with $f+ (u\cdot \nabla) u$, we infer that
  \begin{equation*}
    |\D^2_{22}u_1|\leq C\Big( |D_\ast^2 u| 
    +\frac{1}{2\nu_0}\Big[|\D_1\pi| + |u_t|+\big|(u\cdot\nabla)u\big|
    +|f|_2^2\Big]\Big)\quad \text{a.e. in } \Omega_T.
  \end{equation*}
  Then, squaring, multiplying by $t$, and integrating over
  $(0,t)\times\Omega$, we find
  \begin{equation} 
    \label{eq:u22}
    \begin{aligned}
      &\int_0^t s\,\|\D^2_{22} u_1 (s)\|^2\,ds
      \\
      & \leq \frac{C}{\nu_0}\int_0^ts\,\big( \|D_\ast^2 u (s)\|^2_2 +
      % \frac{1}{2\nu_0}\Big[
      \|\D_1\pi(s)\|^2_2+\|u_t(s)\|^2_2
      +\|(u(s)\cdot\nabla)\,u(s)\|^2_2%\Big]
      \big)\,ds,
    \end{aligned}
  \end{equation}
  To control $\int_0^ts\,\|\D_1\pi(s)\|^2_2\,ds$ we proceed again as
  in Step~2 of the proof of
  Lemma~\ref{lem:main-theo-with-delta}.  Thus, for
  a.e. $t\in[0, T]$, the following inequality holds true
  \begin{equation*}
    \begin{aligned}
      &     \int_0^t s\,\|\D_1\pi (s)\|_2^2\, ds 
      \\
      &\leq C\int_0^t s\,
      \big( \|u_t(s)\|_2^2 + \|\D_1 \nabla u (s)\|_2^2 +
      \delta^{p-2}\mathcal{I}_1(u)(s) +\|f(s)\|_2^2+\|(u(s)\cdot\nabla)\,u(s)\|^2_2 \big)\,ds
      \\
      % & \leq C\bigg(\int_0^t s\, \|(\us\cdot \nabla)\us\|_2^2 \,ds +
      % \|u_0\|_2^2 + \int_0^t \|f(s)\|_2^2\, ds\bigg),
      &\leq C\Big[\|u_0\|_2^2+\int_0^t s\,\big( \|(\us\cdot
      \nabla)\,\us\|_2^2 +\|f (s)\|^2_2\big)\, ds\Big],
    \end{aligned}
  \end{equation*}
  where we have used relations~\eqref{eq:nabla+1}
  and~\eqref{eq:time-space}.  Once again we apply~\eqref{eq:nabla+1},
  so that relation~\eqref{eq:u22} gives, for a.e.  $t\in [0, T]$
  \begin{equation*}
    \begin{aligned}
      & \int_0^t s\, \| \D^2_{22} u_1 (s) \|^2_2\, ds
  %    \\
%      & \leq C \int_0^t s\,\Big( \|D_\ast^2 \us\|^2_2 +
%      \delta^{p-2}\|f(s)\|_2^2 + \|(\us\cdot\nabla)\us\|^2_2 \Big)ds +
%      C\int_0^t\| f(s)\|_2^2 \,ds
%      \\
%      &\leq C \int_0^ts\,\Big( (c + \delta^{p-2}) \| f (s)\|_2^2 +
%      \|(\us\cdot\nabla)\us\|^2_2\Big) ds + C\int_0^t\| f(s)\|_2^2
%      \,ds
%      \\
%      & \leq C\bigg(\int_0^t s\, \|(\us\cdot \nabla)\us\|_2^2\, ds +
%      \|u_0\|_2^2 + \int_0^t \|f(s)\|_2^2\, ds\bigg),
      \leq C\Big[\|u_0\|_2^2+\int_0^t s\,\big( \|(\us\cdot
      \nabla)\,\us\|_2^2 +\|f (s)\|^2_2\big)\, ds\Big],    
    \end{aligned}
  \end{equation*}
  where $C$ depends on $p$, $\delta$, $\nu_0$, $\nu_1$, $T$, and
  $\Omega$.  Summing up the above inequality with~\eqref{eq:nabla+1}
  and~\eqref{eq:time-space}, we get for a.e.  $t\in [0, T]$
  \begin{equation*}
    % \label{eq:nabla+1-and-u22-finalissima}
    \begin{aligned}
      t\,\|\nabla \ut \|_2^2 &+ \nu_0\int_0^ts\,\|D^2 \us\|^2_2 \,ds
      \\
      &\leq C\Big[\|u_0\|_2^2+\int_0^t s\,\big( \|(\us\cdot
      \nabla)\,\us\|_2^2 +\|f (s)\|^2_2\big)\, ds\Big],     
%
%     &\leq c\bigg(\int_0^ts\, \|(\us\cdot\nabla)\us\|^2_2\,ds +
%      \|u_0\|_2^2 + \int_0^t \|f (s)\|_2^2\,ds\bigg),
    \end{aligned}
  \end{equation*}
  whit $C$ depending on $p$, $\delta$, $\nu_0$, $\nu_1$, $T$, and
  $\Omega$. The convective term can be estimated as
  in~\eqref{eq:nonlin-stima-1} and choosing $\varepsilon >0$ small
  enough we get, for a.e. $t\in[0,T]$,
  \begin{equation*}
    \begin{aligned}
  &    t\,\| \nabla \ut \|_2^2 + {\nu_0}\int_0^ts\,\|D^2\us\|^2_2\,ds
      \\
      &\quad \leq c_\varepsilon\int_0^t \big(s\,\|\nabla \us\|_2^2\big)\,
      \|\nabla \us\|_2^2\,ds + C(p,\delta, \nu_0,\nu_1,
      \|f\|_{L^{2}(0, T;L^2)}, \|u_0\|_{2}, T,\Omega).
   \end{aligned}
  \end{equation*}
  Hence, by using Gronwall inequality over $[\lambda,T]$ (for any
  $\lambda>0$), letting
  $\lambda\to 0^+$, and by using the energy inequality we get
  \begin{equation*}
    \int_0^Tt\,\|D^2 u(t)\|^2_2 \,dt \leq
    C(p,\delta, \nu_0,\nu_1, \|f\|_{L^{2}(0, T;L^2)}, \|u_0\|_{2}, T,\Omega).
  \end{equation*}
  Then, the assertion follows by means of
  Theorem~\ref{thm:Dashti-Robinson}.
\end{proof}
\section*{Acknowledgments}
The second author would like to thank G.~Modica for valuable comments
and discussions.

\end{document}